\documentclass[letterpaper, 10 pt, conference]{ieeeconf}  % Comment this line out if you need a4paper
\IEEEoverridecommandlockouts                              % This command is only needed if 
\usepackage{graphics} % for pdf, bitmapped graphics files
\usepackage{mathptmx} % assumes new font selection scheme installed
\usepackage{amsmath} % assumes amsmath package installed
\usepackage{amssymb}  % assumes amsmath package installed
\usepackage[noadjust,sort]{cite}

\usepackage{array,arydshln}

\DeclareSymbolFont{bbold}{U}{bbold}{m}{n}
\DeclareSymbolFontAlphabet{\mathbbold}{bbold}
\newcommand{\onev}{\mathbbold{1}}

\title{\LARGE \bf
State-feedback {S}tabilization of Markov Jump {L}inear {S}ystems\\
with {R}andomly {O}bserved Markov {S}tates
}

\author{Masaki Ogura$^{1}$ and Ahmet Cetinkaya$^{2}$% <-this % stops a space
\thanks{$^{1}$M.~Ogura is with the Department of 
Mathematics and Statistics, Texas Tech University, TX 79409, USA. 
        {\tt\small msk.ogura@gmail.com}, {\tt\small masaki.ogura@ttu.edu}}%
\thanks{$^{2}$Ahmet Cetinkaya is with the Department of Mechanical
and Environmental Informatics, Tokyo Institute of Technology,
Tokyo 152-8552, Japan. {\tt\small ahmet@dsl.mei.titech.ac.jp}}%
}

\newtheorem{theorem}{Theorem}[section]

\newtheorem{proposition}[theorem]{Proposition}
\newtheorem{lemma}[theorem]{Lemma}

\newtheorem{definition}[theorem]{Definition}

\newtheorem{example}[theorem]{Example}
\newtheorem{problem}[theorem]{Problem}

\newtheorem{remark}[theorem]{Remark}

\newenvironment{sketchproof}[1]{% %Extended proof environment
\begin{proof}}{\end{proof}
}

\usepackage{graphicx}

\usepackage{mathtools}
\mathtoolsset{showonlyrefs=true}
\mathtoolsset{centercolon}

\newcommand{\afterequation}{\vskip 3pt}

\newcommand{\norm}[1]{\lVert #1 \rVert}

\renewcommand{\Pr}{\mathcal P}
\newcommand{\oneto}[1]{[#1]}

\begin{document}

\maketitle
\thispagestyle{empty}
\pagestyle{empty}

\begin{abstract}
In this paper we study the state-feedback stabilization of a
discrete-time Markov jump linear system when the observation of the
Markov chain of the system, called the Markov state, is
time-randomized by another Markov chain. Embedding the Markov state
into an extended Markov chain, we transform the given system with
time-randomized observations to another one having the enlarged
Markov-state space but with so-called cluster observations of Markov
states. Based on this transformation we propose linear matrix
inequalities for designing stabilizing state-feedback gains for the
original Markov jump linear systems. The proposed method can treat
both periodic observations and many of renewal-type observations in a
unified manner, which are studied in the literature using different
approaches. A numerical example is provided to demonstrate the
obtained result.
\end{abstract}

\newcommand{\flo}[1]{\lfloor #1 \rfloor_T}

\section{Introduction}

Markov jump linear systems is a class of switched linear systems whose
switching is governed by a time-homogeneous Markov process, called the
Markov state, and have been attracting continuing attention due to its
simplicity as well as its ability of modeling systems in application
such as robotic systems~\cite{Siqueira2004,Vargas2013a},
economy~\cite{Costa2007,Barthelemy2013a}, and networked
systems~\cite{Hespanha2007a}. It is known that, under the assumption
that controllers can observe the Markov state at any time instants,
we can perform standard types of controller synthesis for Markov jump
linear systems such as state-feedback stabilization, quadratic optimal
control, $H^2$ optimal control, and $H^\infty$ optimal control~(see,
e.g., the monograph~\cite{Costa2005}).

However it is often not realistic to assume that controllers {\it
always} have an access to the Markov state and this fact has been
motivating the investigation of the effect of limited or uncertain
observations of the Markov state. For example, the authors
in~\cite{DoVal2002} study the stabilization and $H^2$-control of
discrete-time Markov jump linear systems when the Markov-state space
is partitioned into subsets, called clusters, and an observation of
the Markov state only tells us to which cluster the Markov-state
belongs. Similar studies in the continuous-time settings can be found
in~\cite{Liu2006,Li2006b}. Under an extreme situation when the
Markov-state space has only one cluster, i.e., when one cannot observe
the Markov state, Vargas et al.~\cite{Vargas2013} investigate
quadratic optimal control problems. 

Another but not the only source of uncertainty comes from the
randomness of the time instants at which one can observe the Markov
states. For the case when observation times follow a renewal process,
the authors in~\cite{Cetinkaya2013} design almost-surely stabilizing
state-feedback controllers whose gains are reset whenever an
observation is performed. For the special case when observations are
performed periodically, the same
authors~\cite{Cetinkaya2013a,Cetinkaya2014} derive stabilizing (in the
mean square sense) state-feedback controllers using {Lyapunov-like
functions}. We here remark that other various methods such as, for
example, adaptive strategies~\cite{Bercu2009} could be used to study
this type of problems, though we do not give a detailed survey of the
field in this paper.

In this paper we propose a unified method for designing stabilizing
state-feedback gains for a discrete-time Markov jump linear system
when the time instants at which a controller performs an observation
of the Markov state, called an observation process, is
time-randomized by another Markov chain. This Markov chain can be used
to model various types of observation processes including periodic
observations \cite{Cetinkaya2013a,Cetinkaya2014} and observations
following a renewal process~\cite{Cetinkaya2013}. By embedding the
original Markov-state to another one, we transform a Markov jump
linear system with time-randomized observations to another one with
clustered observations, for which we apply the result in
\cite{DoVal2002} and derive linear matrix inequalities for finding
stabilizing state-feedback gains.

This paper is organized as follows. After preparing the notations used
in this paper, in Section~\ref{sec:MJLS} we give a brief overview of
Markov jump linear systems and their stabilization. Then in
Section~\ref{sec:partObsv} we formulate the stabilization problem with
time-random observation of the Markov state. After showing an
embedding of the Markov state to another Markov chain in
Section~\ref{sec:hidden}, we in Section~\ref{sec:design} derive linear
matrix inequalities for the design of feedback gains. 

\subsection{Mathematical Preliminaries}

The notation used in this paper is standard. Let $\mathbb{N}$ denote
the set of nonnegative integers. Let $\mathbb{R}^n$ and
$\mathbb{R}^{n\times m}$ denote the vector spaces of real $n$-vectors
and $n\times m$ matrices, respectively. By $\norm{\cdot}$ we denote
the Euclidean norm on~$\mathbb{R}^n$. $\Pr(\cdot)$ will be used to
denote the probability of an event. The probability of an event
conditional on an event $\mathcal E$ is denoted by $\Pr(\cdot \mid
\mathcal E)$. Expectations are denoted by $E[\cdot]$. Characteristic
functions are denoted by $\onev(\cdot)$. For a positive integer $N$ we
define the set $\oneto N = \{1, \dotsc, N\}$. For a positive integer
$T$ define $\lfloor k \rfloor_T$ as the unique integer in $\oneto T$
such that $k - \lfloor k\rfloor_T$ is an integer multiple of $T$. 
When a real symmetric matrix $A$ is positive definite we write $A>0$.

\section{Markov Jump Linear Systems and Stabilization}\label{sec:MJLS}

The aim of this section is to give a brief overview of Markov jump
linear systems in discrete-time~\cite{Costa2005} and also recall some
basic definitions of their stability and stabilizability.

Let $n$, $m$, and $N$ be positive integers. Let $A_1, \dotsc, A_N \in
\mathbb{R}^{n\times n}$ and $B_1, \dotsc, B_N \in \mathbb{R}^{m\times
n}$. Also let $r = \{r(k) \}_{k=0}^\infty$ be a time-homogeneous
Markov chain taking its values in $\mathfrak X = \oneto N$ and having
the transition probability matrix~$P\in \mathbb{R}^{N\times N}$. We
call the stochastic difference equation
\begin{equation}\label{eq:MJSL_input}
\Sigma :
x(k+1) = A_{r(k)} x(k) + B_{r(k)} u(k)
\end{equation}
a Markov jump linear system~\cite{Costa2005}. We call $\mathfrak X$
the {\it Markov-state space} of $\Sigma$. Both the initial state $x(0)
= x_0 \in \mathbb{R}^n$ and the initial Markov state $r(0) = r_0\in
\mathfrak X$ are assumed to be constants. The (internal) mean square
stability of $\Sigma$ is defined in the following standard way.

\begin{definition}\label{defn:mss:usual}
$\Sigma$ is said to be {\it mean square stable} if there exist $C>0$
and $\epsilon>0$ such that the solution $x$ of \eqref{eq:MJSL_input}
satisfies
\begin{equation}\label{eq:MSS}
E[\norm{x(k)}^2] < C\epsilon^k\norm{x_0}^2
\end{equation}
for all $x_0$ and $r_0$, provided $u = 0$.
\end{definition}

In this paper we mainly discuss the stabilization of $\Sigma$ via
state-feedback controllers. If one assumes that the controller has an
exact access to the Markov state~$r(k)$ at each time $k\geq 0$, then
we can consider the following mode-dependent controller of the form
\begin{equation}\label{eq:state-fb-normal}
u(k) = K_{r(k)} x(k)
\end{equation}
where $K_1, \dotsc, K_N \in \mathbb{R}^{m\times n}$. We say that the
state-feedback controller~\eqref{eq:state-fb-normal} {\it stabilizes}
$\Sigma$ if the following Markov jump linear system without input
\begin{equation}\label{eq:MJLS_ABK}
x(k+1) = \left( A_{r(k)} + B_{r(k)} K_{r(k)}\right) x(k)
\end{equation}
is mean square stable. 

Another scenario, which is closely related to the current paper, is
the stabilization with so-called cluster observations of Markov
states~{(see }\cite{DoVal2002}{)}. {Following
the notation in \cite{DoVal2002}, we assume} that the Markov state
space~$\mathfrak X$ is decomposed as $\mathfrak X = \mathfrak X_h
\times \mathfrak X_o$, where $\mathfrak X_h$ and $\mathfrak X_o$ are
sets. Thus each \mbox{$i \in \mathfrak X$} can be represented as $i =
(i_h, j_o)$ by some $i_h \in \mathfrak X_h$ and $j_o \in \mathfrak
X_o$. The set $\mathfrak X_h$ ($\mathfrak X_o$) represents the
unobservable (observable, respectively) part of the Markov state
space~$\mathfrak X$. Let us define the projection~$\pi_o\colon
\mathfrak X \to \mathfrak X_o$ by $\pi_o(i_h, j_o) = j_0$. Then, the
state-feedback controller that can observe only the observable part of
the Markov state must take the form
\begin{equation}\label{eq:state-fb:cluster}
u(k) = K_{\pi_o(r(k))} x(k), 
\end{equation}
where $K_j \in \mathbb{R}^{m\times n}$ for each $j \in \mathfrak X_o$.
We say that this feedback controller stabilizes $\Sigma$ if the
solution~$x$ of the closed loop equation $x(k+1) = \left( A_{r(k)} +
B_{r(k)} K_{\pi_o(r(k))}\right) x(k)$ satisfies the condition in
Definition~\ref{defn:mss:usual}.

The following proposition~\cite[Theorem~6]{DoVal2002} gives linear
matrix inequalities whose solutions yield stabilizing feedback gains
for feedback control~\eqref{eq:state-fb:cluster} with clustered
observations. In order to state the proposition, for $i \in \mathfrak
X$ and a family of matrices~$\{R_i\}_{i\in \mathfrak X}\subset
\mathbb{R}^{n\times n}$ we define the matrix $\mathcal D_i(R) \in
\mathbb{R}^{n\times n}$ by {$\mathcal D_i(R) = \sum_{j=1}^N
p_{ji}R_j$.}

\begin{proposition}[{\cite[Theorem~6]{DoVal2002}}]\label{prop:doval}
Assume that the matrices $R_i \in \mathbb{R}^{n\times n}$, $G_j \in
\mathbb{R}^{m\times n}$, and $F_j\in\mathbb{R}^{m\times n}$ ($i\in
\mathfrak X_j$, $j\in \mathfrak X_o$) satisfy the matrix linear
inequalities
\begin{equation}
\begin{bmatrix}
R_i & A_iG_j + B_i F_j
\\
G_j^\top A_i^\top + F_j^\top B_i^\top & G_j + G_j^\top - \mathcal D_i(R)
\end{bmatrix} > 0
\end{equation}
for all $i\in \{j\} \times \mathfrak X_o$ and $j\in \mathfrak X_o$.
Define $K_j = F_j G_j^{-1}$ for each $j \in \mathfrak X_o$. Then the
feedback controller~\eqref{eq:state-fb:cluster} stabilizes $\Sigma$.
\end{proposition}

\section[Random Observation Processes
Induced by Markov Chains]{{Random Observation Processes
Induced\\by Markov Chains}}\label{sec:partObsv}

The aim of this section is to state the stabilization problem with
time-randomly observed Markov states. We in particular introduce a
novel class of random observation processes induced by Markov chains.
{In particular the class contains observation processes that are
not renewal processes.} Let us begin with the next general definition.

\begin{definition}
An $\mathbb{N}$-valued increasing stochastic process $t =
\{t_i\}_{i=0}^\infty$ is called an {\it observation process}.
\end{definition}

Observation processes will be used to model the times at which a
controller can access the Markov state. Given an observation process
$t$, define the stochastic process~\mbox{$\tau =
\{\tau(k)\}_{k=0}^\infty$} by
\begin{equation}
\tau(k) = 
\begin{cases}
\max\{ t_i: t_i \leq k\} & k\geq \max(0, t_0)
\\
\tau_0 &\text{otherwise}
\end{cases}
\end{equation}
where $\tau_0 <0$ is an arbitrary integer. This $\tau(k)$ represents,
for each time~$k$, the most recent time the Markov state was observed.
In particular we have $\tau(t_i) = t_i$ for every $i\geq 0$. Notice
that we augment the process with the arbitrary negative integer
$\tau_0$ when $k<\max(0, t_0)$ because, before the time $k = t_0$,
{no observation is
performed yet}.

We then define another stochastic process~$\sigma =
\{\sigma(k)\}_{k=0}^\infty$ taking its values in $\oneto N$ by
\begin{equation}
\sigma(k) = 
\begin{cases}
r(\tau(k)) & k\geq \max(0, t_0)
\\
\sigma_0 &\text{otherwise}
\end{cases}
\end{equation}
where $\sigma_0 \in \oneto N$ is arbitrary. 
\begin{figure}[tb]
\vspace{.1in}
\centering \includegraphics[width=6.01cm]{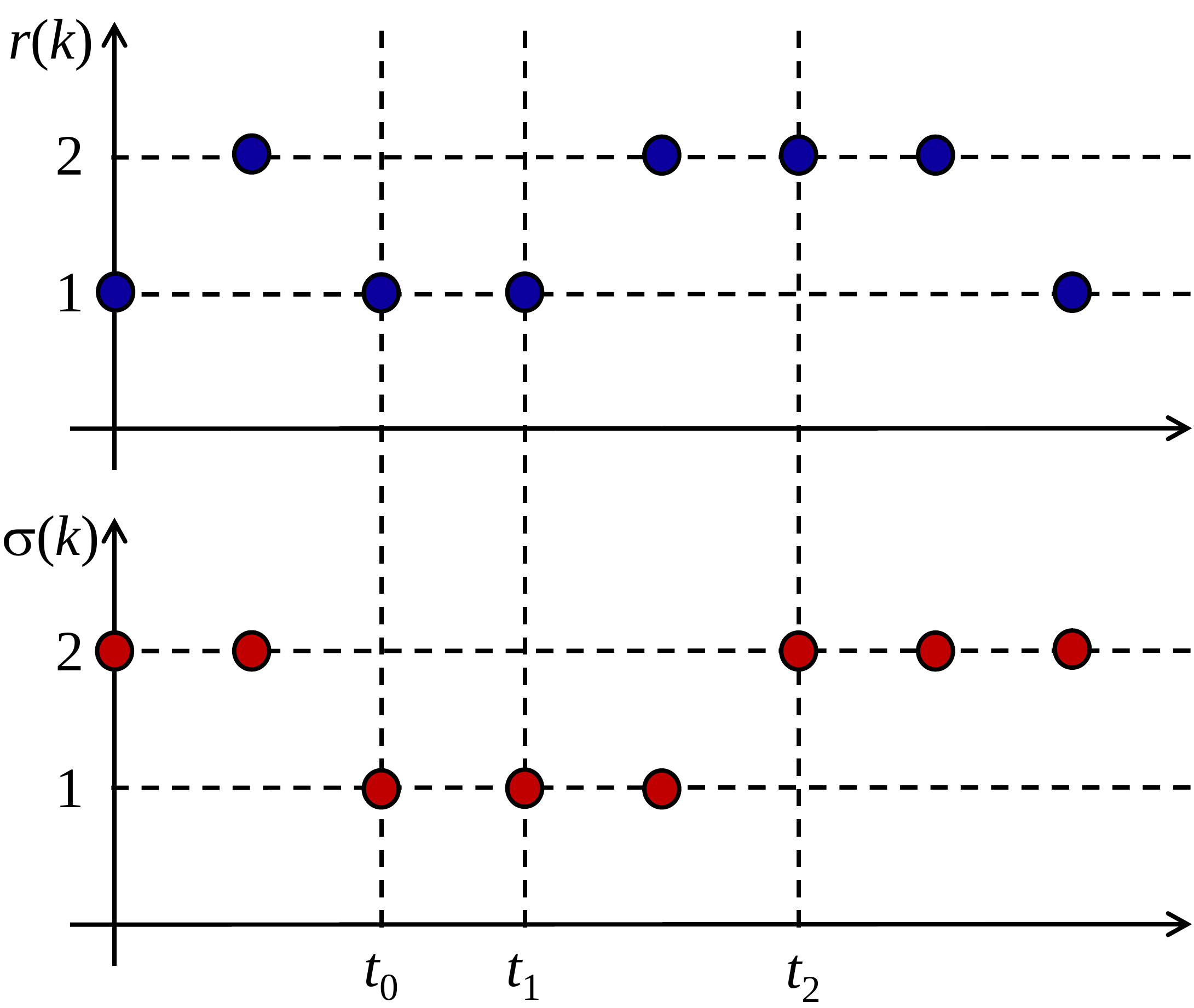}
\caption{Markov state $r$ and its observed version~$\sigma$. Until the
first observation time~$t_0$, the observed version is temporarily set
to $\sigma_0 = 2$.} \label{fig:r_and_signa}
\vspace{-3mm}
\end{figure}
The process $\sigma$ represents the most-updated information of the
Markov state that is available for a controller. We again notice that,
by the same reason as above, the process $\sigma$ is
{augmented} by an arbitrary $\sigma_0$ before the
time $k=\max(0, t_0)$, i.e., before the first observation is
performed. See Fig.~\ref{fig:r_and_signa} for an illustration.

In this paper we assume that the controller has an access to, at each
time $k\geq 0$, the state variable~$x(k)$, the most recent
observation~$\sigma(k)$ of the Markov state~$r$, and $k-\tau(k)$,
which is the time elapsed since the last observation. Then we
construct the state-feedback controller of the form
\begin{equation}\label{eq:state-fb}
u(k) = K_{\sigma(k), \lfloor k+1-\tau(k) \rfloor_T} x(k){,}
\end{equation}
where $K_{\gamma, \delta} \in \mathbb{R}^{m\times n}$ for each $\gamma
\in \oneto N$ and $\delta \in \oneto T$. The first argument
$\sigma(k)$ {in \eqref{eq:state-fb}} allows the gain to be reset
whenever a controller performs an observation of the Markov state as
in~\cite{Cetinkaya2013,Cetinkaya2013a,Cetinkaya2014}. The second
argument allows the controller to change feedback gains between two
consecutive observations rather than keeping them to be constant,
which can enhance the performance of the
controller~\cite{Cetinkaya2014}. The reason for taking the operator
$\lfloor \cdot \rfloor_T$ in the second argument of $K$ is that,
otherwise, we have to design infinitely many matrices $K_{\gamma,
\delta}$ where $\delta$ could be any nonnegative numbers. Taking the
operator $\lfloor \cdot \rfloor_T$ forces $\delta$ to be in the finite
set $\oneto T$, which turns out to make our stabilization problem
solvable in finite time.

Combining \eqref{eq:MJSL_input} and \eqref{eq:state-fb} we obtain the
closed loop equation
\begin{equation}\label{eq:Sigma_K}
\Sigma_{K} 
:
x(k+1) 
= 
\left(A_{r(k)} + B_{r(k)} K_{\sigma(k), \lfloor k+1-\tau(k) \rfloor_T}\right) x(k). 
\end{equation}
Extending Definition~\ref{defn:mss:usual}, we define the mean square
stability of the system $\Sigma_{K}$ as follows.

\begin{definition}\label{defn:}
Let $\mathcal T$ be a set of observation processes. We say that the
pair $(\Sigma_{K}, \mathcal T)$ is {\it mean square stable} if there
exist $C>0$ and $\epsilon\in [0, 1)$ such that the solution~$x$ of
\eqref{eq:Sigma_K} satisfies \eqref{eq:MSS} for all $x_0\in
\mathbb{R}^n$, $r_0\in \oneto N$, $\tau_0<0$, $\sigma_0\in \oneto N$,
and~\mbox{$t\in \mathcal T$}. The feedback control \eqref{eq:state-fb}
is said to stabilize $(\Sigma, \mathcal T)$ if $(\Sigma_{K}, \mathcal
T)$ is mean square stable.
\end{definition}

\subsection{Observation Process Induced by Markov Chains}
 
In this paper we deal with a class of observation processes induced by
time-homogeneous Markov chains. In order to introduce the class, we
first need to define observation processes induced by deterministic
sequences. Let $s\colon \mathbb{N} \to \oneto M$ be an arbitrary
sequence and let $\Lambda$ be a subset of $\oneto M$. Assume that $s$
intersects with $M$ infinitely many times, namely, that the set
$\{k\in\mathbb{N}: s(k)\in \Lambda\}$ is infinite. Then define the
infinite sequence~$t_{\Lambda}(s)$ as the one obtained by increasingly
ordering the numbers in the infinite set~$\{k\in\mathbb{N}: s(k)\in
\Lambda\}$. Thus, the sequence $t_{\Lambda}(s)$ consists of the times
$k$ at which the sequence $s$ intersects with $\Lambda$. {For example,
the observation time instants~$t_0$, $t_1$, and $t_2$ shown in
Fig.~\ref{fig:r_and_signa} are induced by the sequence~$s$ shown in
Fig.~\ref{fig:s_and_t} with $\Lambda = \{2\}$.}
\begin{figure}[tb]
\vspace{.1in}
\centering \includegraphics[width=6.01cm]{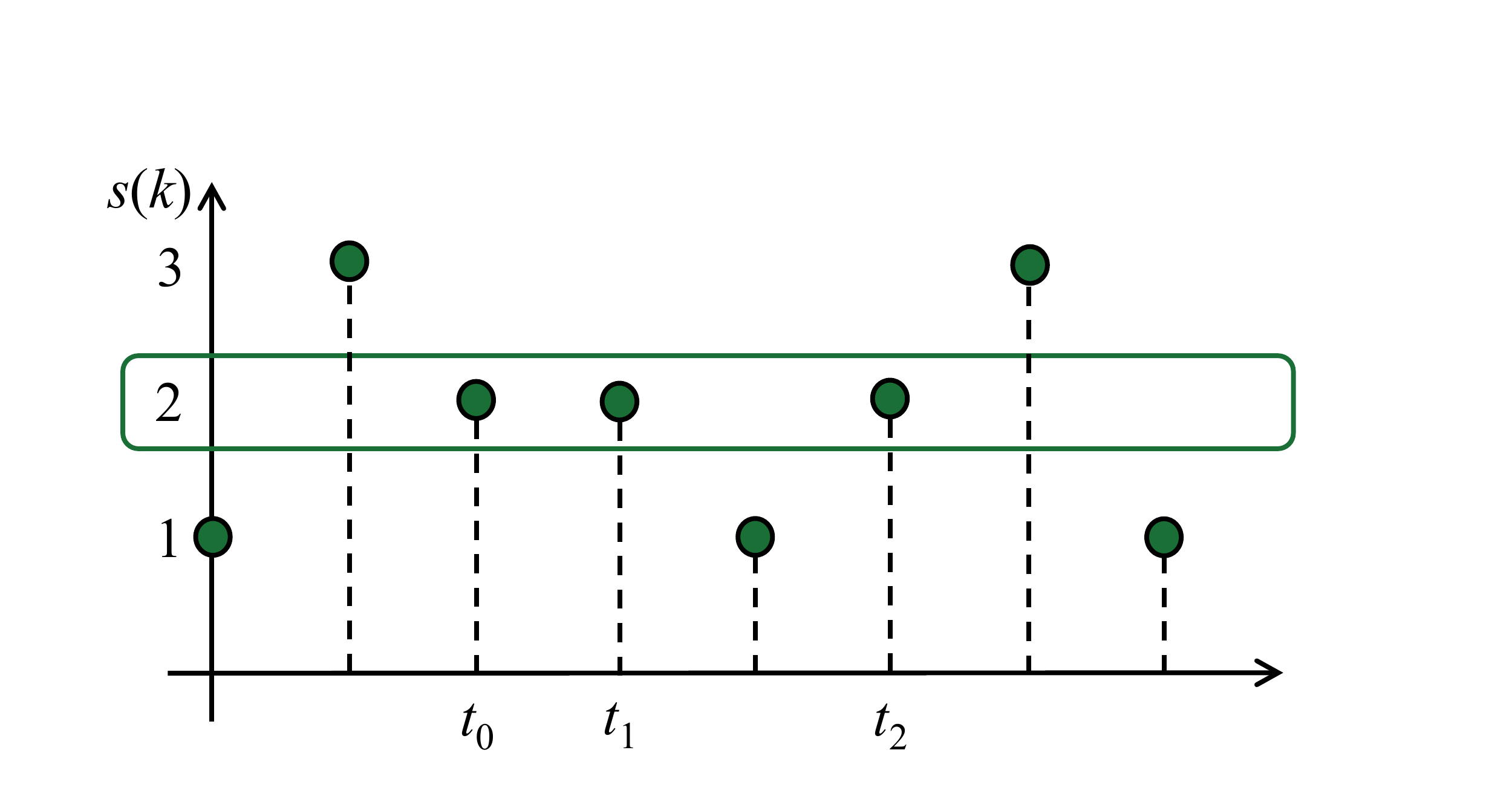} \caption{A
sequence $s$ and the induced observation time instants $t$. An
observation occurs whenever $s$ enters the set $\Lambda = \{2\}$.}
\label{fig:s_and_t}
\vspace{-3mm}
\end{figure}

Then we {extend} the above definition to Markov chains as follows. Let
$s$ be a time-homogeneous Markov chain taking its values in $\oneto M$
and let {$\Lambda\subset \oneto M$} be a nonempty {set that is
recurrent with respect to the Markov chain~$s$}. We define a family of
observation processes $\mathcal T_{s,\Lambda}$ by
\begin{equation}\label{eq:def:T_s}
\mathcal T_{s, \Lambda} = \{ t_\Lambda(s(\cdot;s_0)) \}_{s_0=1}^M, 
\end{equation}
where $s(\cdot;s_0)$ denotes the Markov chain~$s$ when its initial
state equals $s_0$. {Notice that, since $\Lambda$ is recurrent,
$s$ intersects with $\Lambda$ infinitely many times with probability
one and thereby $\mathcal T_{s, \Lambda}$ is well defined.} The
following examples illustrate that the family $\mathcal T_{s,
\Lambda}$ can express various types of observation processes.

\begin{example}[Periodic observation with failures]\label{ex:per:fail}
Let $\tau$ be a positive integer. Let $\Lambda = \{1\}$ and the
transition probability matrix of $s$ be
\begin{equation}
\left[\begin{array}{cc:ccc}
&&1\\\hdashline
&&1\\
&&&\ddots\\
&&&&1\\ \hdashline
p&1-p
\end{array}
\right]\in \mathbb{R}^{(\tau+1)\times (\tau+1)}, 
\end{equation}
where zero entries are omitted. Then we can see that, if $t\in
\mathcal T_{s, \Lambda}$, then the difference $t_{i+1} - t_i$ of
observation times independently follow the distribution $\mu$ on
$\mathbb{N}$ that is concentrated on the set $\{\tau, 2\tau, 3\tau,
\dotsc \}$ and satisfies \mbox{$\mu(\{\tau k\}) = (1-p)^{k-1}p$} for
every $k\geq 1$. In other words, this observation process expresses
the observation of every $\tau$ time units with the probability of
failure $1-p$ at each observation. In particular, if $p=1$, then the
observation process gives the observation with period $\tau$, which is
considered in~\cite{Cetinkaya2014}.
\end{example}

\begin{example}[Renewal processes]
Let $\mu$ be an arbitrary distribution on $\mathbb{N}\backslash \{0\}$
having finite support. Then there exist $\tau>0$ and
$\{p_k\}_{k=1}^\tau$ such that $\sum_{k=1}^\tau p_k=1$ and $\mu(\{ k
\}) = p_k$ for every $1\leq k\leq \tau$. Define the positive
integers~$\tilde p_1$, $\dotsc$, $\tilde p_\tau$ recursively by
{$\tilde p_k = {p_k}/{\prod_{\ell=1}^{k-1} (1-\tilde p_\ell)}$ for
$k=1, \dotsc, \tau$.} Let $s$ be the Markov chain having the
transition probability matrix
\begin{equation}
\left[
\begin{array}{c:ccc}
\tilde p_1&1-\tilde p_1
\\
\vdots &&\ddots
\\
\tilde p_{\tau-1}&&&1-\tilde p_{\tau-1}
\\ \hdashline
1
\end{array}
\right] \in \mathbb{R}^{\tau \times \tau }, 
\end{equation}
where zero entries are omitted again. {Also let $\Lambda =
\{1\}$.} We can see that, if $t\in \mathcal T_{s, \Lambda}$, then
difference $t_{i+1} - t_i$ of observation times independently follow
the distribution $\mu$ {and therefore $t$ forms a renewal
process}.
\end{example}

{Finally we present a simple example of observation
processes~$\mathcal T_{s, \Lambda}$ that are not renewal processes and
therefore cannot be treated by the method in~\cite{Cetinkaya2013}.}

\begin{example}
{Let the transition probability matrix of $s$ be
\begin{equation}
\begin{bmatrix}
0 & 1 & 0
\\
0 & 0 & 1
\\
1 & 0 & 0
\end{bmatrix}
\end{equation}
and also let $\Lambda = \{1, 3\}$. If $s_0 = 1$, then we have
$\{t_i\}_{i=0}^\infty = \{0, 2, 3, 5, 6, 8, 9, \dotsc\}$, which cannot
be a renewal process because the differences of two consecutive
observation times equal $\{2, 1, 2, 1, 2, \dotsc\}$ almost surely.}
\end{example}

Now we state the main problem studied in this paper.

\begin{problem}\label{prb:}
Given a Markov jump linear system~$\Sigma$ and a family of observation
processes $\mathcal T_{s, \Lambda}$ induced by a time-homogeneous
Markov chain~$s$, find feedback gains \mbox{$K = \{K_{\gamma,
\delta}\}_{\gamma \in \oneto N, \delta \in \oneto T}$} that stabilize
the pair~$(\Sigma, \mathcal T_{s, \Lambda})$.
\end{problem}

\section{Embedding of Markov States}\label{sec:hidden}

The difficulty of solving Problem~\ref{prb:} is that the
system~$\Sigma_K$ in \eqref{eq:Sigma_K} is no longer a standard Markov
jump linear system and thus the techniques established in the
literature~\cite{Costa2005} cannot be used to it. Also, by the generality of
the observation processes~$\mathcal T_{s,\Lambda}$ {discussed in
the previous section}, we also cannot use the results recently
proposed in~\cite{Cetinkaya2013,Cetinkaya2013a,Cetinkaya2014}. The aim
of this section is to show that we can embed the Markov state
process~$r$ to another Markov chain, with which we can express
$\Sigma_K$ as a {standard} Markov jump linear system. This fact
will be used in the next section to reduce Problem~\ref{prb:} to the
stabilization of a Markov jump linear system with clustered
observations.

Let $r$ be the Markov state of the Markov jump linear system~$\Sigma$.
Let $s$ be the time-homogeneous Markov chain that induces the family
$\mathcal T_{s, \Lambda}$ of observation processes. We let $P =
[p_{ij}]_{i,j} \in \mathbb{R}^{N\times N}$ and $Q = [q_{ij}]_{i,j}\in
\mathbb{R}^{M\times M}$ denote the transition probability matrices of
$r$ and $s$, respectively. The next proposition is the main result of
this section.

\begin{proposition}\label{prop:Markov}
Define 
\begin{equation}
\bar{\mathfrak X} = \oneto N \times \oneto M \times \oneto N \times \oneto T. 
\end{equation}
Then the $\bar{\mathfrak X}$-valued stochastic process $\bar r$ defined by
\begin{equation}\label{eq:def:bar r}
\bar r(k) = (r(k), s(k), \sigma(k), \lfloor k+1-\tau(k) \rfloor_T),\ k\geq 0
\end{equation}
is a time-homogeneous Markov chain. Moreover its transition
probabilities are given by, for all $\chi = (\alpha, \beta, \gamma,
\delta)$ and {$\chi' = (\alpha', \beta', \gamma\,', \delta')$} in
$\bar{\mathfrak X}$,
\begin{equation}\label{eq:TransProb.bar-r}
\begin{multlined}
\Pr(\bar r(k+1) = \chi' \mid \bar r(k) = \chi) 
\\
= 
\begin{cases}
\onev(\alpha' = \gamma\,',\,\delta=1) p_{\alpha, \alpha'} q_{\beta,\beta'}
&
\beta'\in \Lambda, 
\\
\onev(\gamma\,' = \gamma,\,\delta' = \lfloor \delta + 1\rfloor_T) p_{\alpha, \alpha'} q_{\beta,\beta'}
&
\beta'\notin \Lambda.
\end{cases}
\end{multlined}
\end{equation}
\afterequation
\end{proposition}

\begin{proof}
Let $k_0\in \mathbb{N}$ and $k\geq k_0$ be arbitrary. Take arbitrary
$\chi_i = (\alpha_i, \beta_i, \gamma_i, \delta_i)\in \bar{\mathfrak
X}$ ($i=k_0, \dotsc, k+1$). For each $i$ define the events $\mathcal
E_i,\,\mathcal F_i$ by $\mathcal E_i
=
\{\bar r(i) = \chi_i,\,\dotsc,\,\bar r(k_0) = \chi_{k_0}\}$ and
\begin{equation}
\begin{aligned}
\mathcal F_i 
&= 
\{ \bar r(i) = \chi_i \}
\\
&=
\{ r(i) =\alpha_i,\,s(i) =\beta_i,\,\sigma(i) = \gamma_i,\,\lfloor i+1-\tau(i)\rfloor_T = \delta_i \}.
\end{aligned}
\end{equation}
Under the assumption that $\mathcal E_k$ is not the null set, we need to
evaluate the conditional probability
\begin{equation}\label{eq:P(r(k+1)=...}
\begin{multlined}
\Pr(\bar r(k+1) = \chi_{k+1} \mid \mathcal E_k)
=
{{{{\Pr(\mathcal E_{k+1})}}/{{\Pr(\mathcal E_k)}}}}. 
\end{multlined}
\end{equation}
Remark that this assumption implies that
{\begin{equation}\label{eq:sigma(k)=gamma_k:merged}
\sigma(k) = \gamma_{\,k}
,\ 
\lfloor k+1-\tau(k) \rfloor_T = \delta_k, 
\end{equation}}
because otherwise $\mathcal E_k$ equals a null set.

First assume that $\beta_{k+1} \in \Lambda$. Then we have \mbox{$s(k+1)\in
\Lambda$} so that, by the definition of $\mathcal T_{s, \Lambda}$, an
observation occurs at time~$k+1$, i.e., we have $\tau(k+1) = k+1$ and
\mbox{$\sigma(k+1) = r(k+1)$}. This implies that
\begin{equation}
\begin{multlined}
\mathcal F_{k+1}
=
\{ r(k+1) = \alpha_{k+1}, s(k+1) = \beta_{k+1}, 
\\ 
\alpha_{k+1} = \gamma_{\,k+1}, [1]_T = \delta_{k+1}\}. 
\end{multlined}
\end{equation}
Therefore{, since $\mathcal E_{k+1} = \mathcal E_k \cap \mathcal
F_{k+1}$, }
{\begin{equation}\label{eq:calEk+1:first}
\begin{multlined}
\mathcal E_{k+1}
=
\{\alpha_{k+1} = \gamma_{\,k+1}, \delta_{k+1} = 1\} \\\hspace{1cm}
\cap \{r(k+1) = \alpha_{k+1}, s(k+1) = \beta_{k+1}\} \cap \mathcal E_k
\end{multlined}
\end{equation}}%
and hence 
\begin{equation}\label{eq:for12.1}
\begin{aligned}
\hspace*{-.5cm}\Pr(\mathcal E_{k+1})
=
\onev(\alpha_{k+1} = \gamma_{\,k+1}, \delta_{k+1} = 1)\hspace{2.3cm} 
\\
 \Pr(\{r(k+1) = \alpha_{k+1}, s(k+1) = \beta_{k+1}\} \cap \mathcal E_k).\hspace*{-.5cm}
\end{aligned}
\end{equation}
The probability appearing in the last term of this equation can be computed as
\begin{equation}\label{eq:for12.2}
\begin{aligned}
&\,
\Pr(\{r(k+1) = \alpha_{k+1}, s(k+1) = \beta_{k+1}\} \cap \mathcal E_k)
\\
=&\,
{\Pr(\mathcal E_k)} \Pr(r(k+1) = \alpha_{k+1}, s(k+1) = \beta_{k+1} \mid \mathcal E_k) 
\\
=&\,
{\Pr(\mathcal E_k)} \Pr(r(k+1) = \alpha_{k+1}, s(k+1) = \beta_{k+1} 
\\
&\hspace{3.5cm}\mid r(k)=\alpha_k, s(k) = \beta_k)
\\
=&\,
{\Pr(\mathcal E_k)}  {p_{\alpha_k,\alpha_{k+1}} q_{\beta_k,\beta_{k+1}}}, 
\end{aligned}
\end{equation}
where we used the fact that both $r$ and $s$ are time-homogeneous
Markov chains. Thus equations \eqref{eq:P(r(k+1)=...},
\eqref{eq:for12.1}, and \eqref{eq:for12.2} conclude that, {for the
case of $\beta_{k+1}\in \Lambda$,}
\begin{equation}\label{eq:hidden1}
\begin{multlined}
\Pr(\bar r(k+1) = \chi_{k+1} \mid \mathcal E_k) 
\\
=
\onev(\alpha_{k+1} = \gamma_{\,k+1}, \delta_{k+1} = 1) p_{\alpha_k,\alpha_{k+1}} q_{\beta_k, \beta_{k+1}}.
\end{multlined}
\end{equation}

{Then consider the case where $\beta_{k+1}\notin \Lambda$.} {In this
case, Markov state $r$ is not observed at time $k+1$} so that we have
$\tau(k+1) = \tau(k)$ and $\sigma(k+1) = \sigma(k)$. Therefore, using
equations {\eqref{eq:sigma(k)=gamma_k:merged}}, in the same way as we
derived \eqref{eq:calEk+1:first} we can show that
\begin{equation}\label{eq:calEk+1:second}
\begin{aligned}
\mathcal E_{k+1}
&=
\{\gamma_{\,k+1} = \gamma_{\,k}, \delta_{k+1} = \lfloor \delta_k+1 \rfloor_T\} 
\\&\hspace{1cm} \cap \{r(k+1) = \alpha_{k+1}, s(k+1) = \beta_{k+1}\} \cap \mathcal E_k
\end{aligned}
\end{equation}
and hence 
\begin{equation}
\begin{aligned}
\Pr(\mathcal E_{k+1})
&=
\onev(\gamma_{\,k+1} = \gamma_{\,k}, \delta_{k+1} = \lfloor \delta_k+1 \rfloor_T) 
\\&\hspace{.9cm} \Pr(\{r(k+1) = \alpha_{k+1}, s(k+1) = \beta_{k+1}\} \cap \mathcal E_k).
\end{aligned}
\end{equation}
Therefore, equations \eqref{eq:P(r(k+1)=...}, \eqref{eq:for12.2}, and
\eqref{eq:calEk+1:second} show that,{for $\beta_{k+1}\notin\Lambda$,}
\begin{equation}\label{eq:hidden2}
\begin{multlined}
\Pr(\bar r(k+1) = \chi_{k+1} \mid \mathcal E_k) 
\\
=
\onev(\gamma_{\,k} = \gamma_{\,k+1}, \delta_{k+1} = \lfloor \delta_k + 1\rfloor_T) 
p_{\alpha_k,\alpha_{k+1}} q_{\beta_k, \beta_{k+1}}.
\end{multlined}
\end{equation}

Since the probabilities \eqref{eq:hidden1} and \eqref{eq:hidden2} do
not depend on $k_0$, letting $k_0 = k$ and $k_0 = 0$ in 
\eqref{eq:hidden1} and \eqref{eq:hidden2} we obtain
\begin{equation}
\begin{multlined}
\Pr(\bar r(k+1) = \chi_{k+1} \mid \bar r(k) = \chi_k, \dotsc,  \bar r(0) = \chi_0)
\\
=
\Pr(\bar r(k+1) = \chi_{k+1} \mid \bar r(k) = \chi_k)
\end{multlined}
\end{equation}
for every $k\geq 0$. This shows that $\bar r$ is a Markov chain since
$\chi_0, \dotsc, \chi_{k+1} \in \bar{\mathfrak X}$ were arbitrarily
taken. Moreover, since the probabilities \eqref{eq:hidden1} and
\eqref{eq:hidden2} do not depend on $k$, we conclude that the Markov
chain $\bar r$ is time-homogeneous and its transition probabilities
are actually given by \eqref{eq:TransProb.bar-r}.
\end{proof}

\begin{remark}
It is observed in \cite{Cetinkaya2014} that the process $\{r(k),
\sigma(k)\}_{k \geq 0}$ itself is indeed a Markov chain but not
time-homogeneous. Proposition~\ref{prop:Markov} shows that, however,
augmenting the third and fourth components in \eqref{eq:def:bar r}
enables us to construct a time-homogeneous Markov chain, which plays a
crucial role in the next section.
\end{remark}

\section{Designing Stabilizing Feedback Gains\\via Linear Matrix Inequalities}\label{sec:design}

In this section we show a set of feedback gains $K$ that stabilizes
$(\Sigma, \mathcal T_{s,\Lambda})$ can be found by solving a set of
linear matrix inequalities. For the proof we will use the reduction of
$\Sigma_K$, which is not necessarily a Markov jump linear system, to a
Markov jump linear system with Markov state evolving in the same way
as the Markov chain~$\bar r$ presented in the last section.

We define $\theta$ as the time-homogeneous Markov chain taking its
values in $\bar{\mathfrak X}$ and having the same transition
probability as $\bar r$, i.e., we assume that the transition
probability $\bar p_{\chi, \chi'}$ of $\theta$ is given
by
\begin{equation}\label{eq:TransProb.theta}
\begin{multlined}
%\Pr(\theta(k+1) = \chi' \mid \theta(k) = \chi) 
%\\
\bar p_{\chi, \chi'}= 
\begin{cases}
\onev(\alpha' = \gamma\,',\,\delta=1) p_{\alpha, \alpha'} q_{\beta,\beta'}
&
\beta'\in \Lambda
\\
\onev(\gamma\,' = \gamma,\,\delta' = \lfloor \delta + 1\rfloor_T) p_{\alpha, \alpha'} q_{\beta,\beta'}
&
\beta'\notin \Lambda
\end{cases}
\end{multlined}
\end{equation}
where $\chi' = (\alpha', \beta', \gamma\,', \delta')$ and $\chi =
(\alpha, \beta, \gamma, \delta)$. The following lemma will be used to
prove the main result of this paper.

\begin{lemma}\label{lem:equvchains}
Let $r_0\in \oneto N$, $s_0\in \oneto M$, $\tau_0<0$, and $\sigma_0\in \oneto
N$ be arbitrary. Then there exists $\chi_0 \in \bar{\mathfrak X}$ such that 
\begin{equation}\label{eq:theta=bar-r}
\theta(k; \chi_{{0}})
=
\bar r (k;r_0, s_0, \sigma_0, \lfloor 1-\tau_0\rfloor_T), 
\end{equation}
where the arguments following the time $k$ denote the initial
conditions of $\theta$ and $\bar r$.
\end{lemma}

{Now,} for the Markov jump linear system $\Sigma$, we
introduce another Markov jump linear system having $\theta$ as its
Markov state as follows. We define the matrices~\mbox{$\bar A_\chi \in
\mathbb{R}^{n\times n}$} and \mbox{$\bar B_\chi \in \mathbb{R}^{n\times m}$}
for each $\chi = (\alpha, \beta, \gamma, \delta)\in \bar{\mathfrak X}$
by $\bar A_\chi = A_{\alpha}$ and $\bar B_\chi = B_{\alpha}$. Then
define the Markov jump linear system $\bar \Sigma$ by
\begin{equation}\label{eq:bar:MJLS}
\bar \Sigma 
:
\bar x(k+1) = \bar A_{\theta(k)} \bar x(k) + \bar B_{\theta(k)} \bar u(k), 
\end{equation}
with the initial states $\bar x(0) = \bar x_0 \in \mathbb{R}^n$ and
$\bar \theta(0) = \bar \theta_0 \in \bar{\mathfrak X}$. Let us also
consider the following standard state-feedback controller
\begin{equation}\label{eq:bar_state-fb}
\bar u(k) = \bar K_{\theta(k)} \bar x(k)
\end{equation}
where $\bar K_\chi \in \mathbb{R}^{m\times n}$ for each $\chi \in
\bar{\mathfrak X}$. Then we obtain the closed loop equation
\begin{equation}
\bar \Sigma_{\bar K}
:
\bar x(k+1) = (\bar A_{\theta(k)} +  \bar B_{\theta(k)} \bar K_{\theta(k)}) \bar x(k).
\end{equation}
The next theorem is the first major result of
this paper.

\begin{theorem}
Assume that $\{\bar K_{\chi}\}_{\chi\in \bar{\mathfrak X}}\subset
\mathbb{R}^{m\times n}$ stabilizes $\bar \Sigma$ and satisfies
\begin{equation}\label{eq:barK:constraint}
\bar K_{\alpha, \beta, \gamma, \delta}
=
\bar K_{\alpha', \beta', \gamma, \delta}
\end{equation}
for all $\alpha, \alpha'\in \oneto N$, $\beta, \beta'\in \oneto M$,
$\gamma \in \oneto N$, and $\delta \in \oneto T$. For each $\gamma \in
\oneto N$ and $\delta \in \oneto T$ define $K_{\gamma, \delta}$ by
\begin{equation}\label{eq:def:barK}
K_{\gamma, \delta} = \bar K_{1, 1, \gamma, \delta}. 
\end{equation}
Then $K$ stabilizes $(\Sigma, \mathcal T_{s, \Lambda})$.
\end{theorem}

\begin{proof}
Assume that $\{\bar K_{\chi}\}_{\chi\in \bar{\mathfrak X}}\subset
\mathbb{R}^{m\times n}$ stabilizes $\bar \Sigma$ and satisfies
\eqref{eq:barK:constraint}. Then there exist $C>0$ and $\epsilon\in
[0, 1)$ such that the solution $\bar x$ of $\bar \Sigma_{\bar K}$
satisfies $E[\norm{\bar x(k)}^2] < C \epsilon^k \norm{\bar x_0}^2$.
Define $K$ by \eqref{eq:def:barK} and let us show that $K$ stabilizes
$(\Sigma, \mathcal T_{s, \Lambda})$. Let $x_0\in \mathbb{R}^n$,
$r_0\in \oneto N$, $s_0\in \oneto M$, $\tau_0<0$, $\sigma_0\in \oneto
N$ be arbitrary. By Lemma~\ref{lem:equvchains}, we can take the
corresponding $\chi_0 \in \bar{\mathfrak X}$ such that
\eqref{eq:theta=bar-r} holds.  Then we can see that
\begin{equation}
\begin{aligned}
x(k+1) 
&= 
\left(
A_{r(k)} + B_{r(k)} K_{\sigma(k), \lfloor k-\tau(k) \rfloor_T}
\right) x(k)
\\
&=
\left(
\bar A_{\theta(k)} + \bar B_{\theta(k)} \bar K_{\theta(k)}
\right) x(k){.}
\end{aligned}
\end{equation}
This shows $\bar x(k) = x(k)$ provided the initial states of $\Sigma$
and $\bar \Sigma$ {coincide} as $x_0 = \bar x_0$.
Therefore, since $\bar \Sigma_{\bar K}$ is mean square stable, we
obtain $E[\norm{x(k)}^2] < C\epsilon^k \norm{x_0}^2$ for every $k$.
Hence $K$ stabilizes $(\Sigma, \mathcal T_{s, \Lambda})$, as desired.
\end{proof}

The constraint \eqref{eq:barK:constraint} leads us to decompose
$\bar{\mathfrak X}$ into the unobservable part~$\bar{\mathfrak X}_h =
\oneto N \times \oneto M$ and the observable
part~\mbox{$\bar{\mathfrak X}_o = \oneto N \times \oneto T$} as
{$\bar{\mathfrak X} = \bar{\mathfrak X}_h \times \bar{\mathfrak
X}_o$.} Then the stabilization of~$\bar \Sigma$ with the feedback
control~\eqref{eq:bar_state-fb} satisfying the
constraint~\eqref{eq:barK:constraint} on feedback gains is equivalent
to the stabilization of $\bar \Sigma$ via clustered
observation~\cite{DoVal2002} reviewed in Section~\ref{sec:MJLS}.
Therefore, using Proposition~\ref{prop:doval} we immediately obtain
the next theorem.

\begin{theorem}\label{thm:LMI}
For $R_\chi \in \mathbb{R}^{n\times n}$ ($\chi \in \bar{\mathfrak X}$)
define\ $\mathcal D_{\chi}(R) = \sum_{\chi'\in\bar{\mathfrak X}} \bar
p_{\chi', \chi} R_{\chi'}$. Assume that $R_\chi \in
\mathbb{R}^{n\times n}$, $G_{\gamma,\delta}\in \mathbb{R}^{n\times
n}$, and $F_{\gamma,\delta}\in \mathbb{R}^{m\times n}$ satisfy the
linear matrix inequality
\begin{equation}\label{eq:lmi}
\begin{bmatrix}
R_\chi 
& 
A_\alpha  G_{\gamma, \delta} + B_{\alpha} F_{\gamma, \delta}
\\
G_{\gamma, \delta}^\top A_\alpha^\top + B_{\alpha}^\top F_{\gamma, \delta}^\top 
& 
G_{\gamma, \delta} + G_{\gamma, \delta}^\top - \mathcal D_\chi(R)
\end{bmatrix} > 0
\end{equation}
for all $ \chi = (\alpha, \beta, \gamma, \delta) \in \mathfrak X$. For
each $\gamma \in \oneto N$ and $\delta \in \oneto T$ define
$K_{\gamma, \delta} = F_{\gamma, \delta}G_{\gamma, \delta}^{-1}$. Then
$K$ stabilizes $(\Sigma, \mathcal T_{s, \Lambda})$.
\end{theorem} 

\begin{example}
Let $N=3$ and consider the Markov jump linear system $\Sigma$ 
given by the matrices
\begin{equation}
\begin{gathered}
A_1 = \begin{bmatrix}
-0.45 &-0.3 \\ 1.2 & 0.45
\end{bmatrix},\ 
A_2 = A_3 = \begin{bmatrix}
-0.7 & 0.7 \\ 0.2& 0.8
\end{bmatrix}{,}
\\
B_1 = \begin{bmatrix}
1 \\ 1
\end{bmatrix},\ 
B_2 = 
\begin{bmatrix}
1 \\ 0
\end{bmatrix},\ 
B_3 = 
\begin{bmatrix}
-1 \\ 0
\end{bmatrix}.
\end{gathered}
\end{equation}
Let the {transition probabilities} of the Markov state~$r$ be {given
by $p_{ii} = 0.6$ for every $i$ and $p_{ij} = 0.2$ for all distinct
$i$ and~$j$.} Assume that a controller tries an observation of the
Markov state every 4 time units, but it fails with probability~$1/2$.
The corresponding Markov chain~$s$ can be realized by letting
\mbox{$\tau =4$} and $p=1/2$ in Example~\ref{ex:per:fail}. {Also we
set $T=4$.} Solving the linear matrix inequalities~\eqref{eq:lmi} we
obtain stabilizing feedback gains. We construct 100 sample paths of
the solution $x$ of the stabilized system~$\Sigma_K$.
Figs.~\ref{fig:graph3} and~\ref{fig:graph3log} show the sample average
and the sample paths of $\norm{x(k)}^2$, respectively. We can see
that, even though the observation of the Markov state is not
{necessarily} performed periodically, the designed controller attains
stabilization.

\begin{figure}[tb]
\vspace{.08in}
\centering 
\includegraphics[width=6cm]{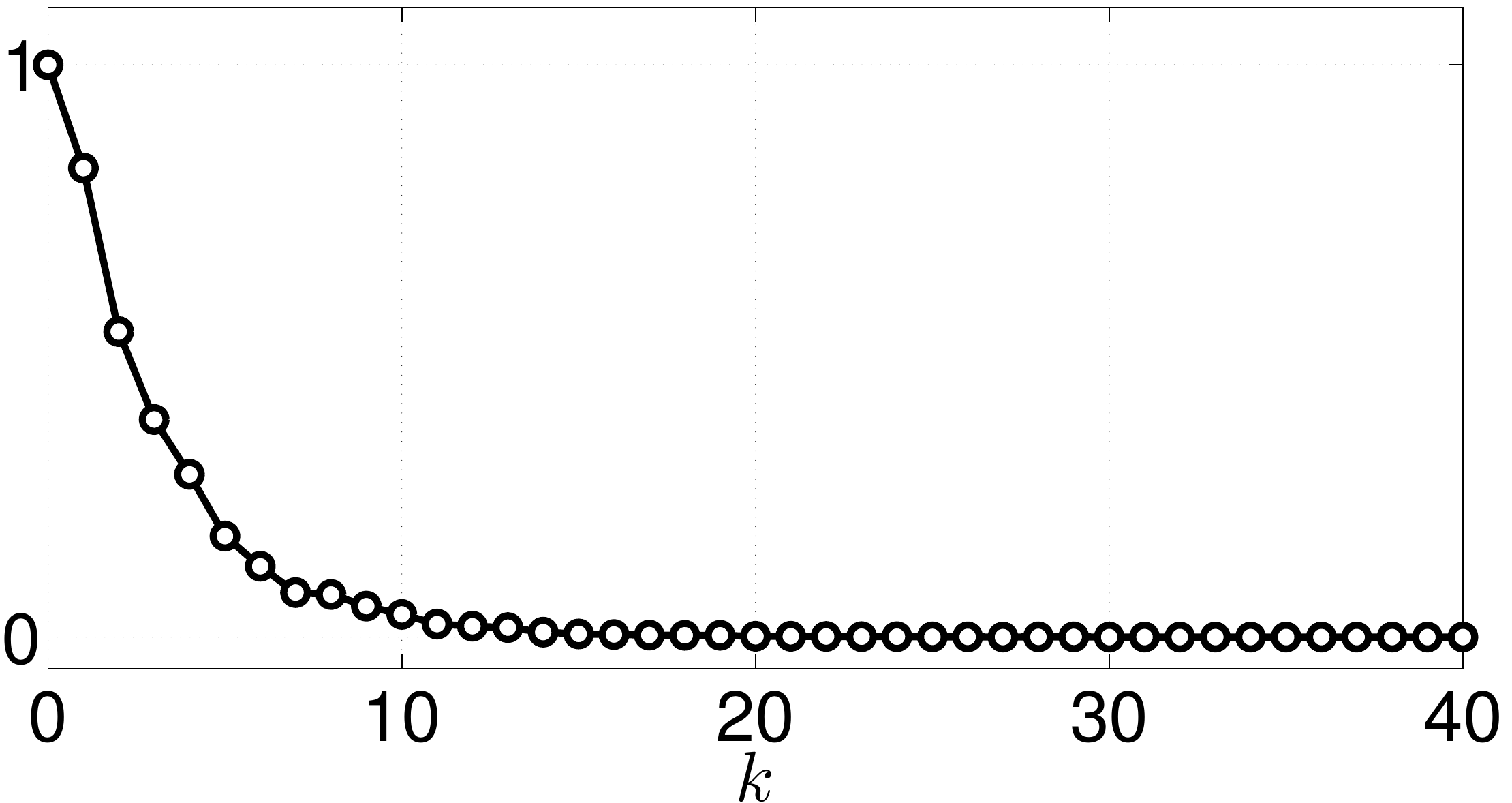}
\caption{Stabilized system:  Sample average of $\norm{x(k)}^2$}
\label{fig:graph3}
\end{figure}

\begin{figure}[tb]
%\vspace*{-3cm}
\centering 
{\hspace*{-.47cm}\includegraphics[width=6.5cm]{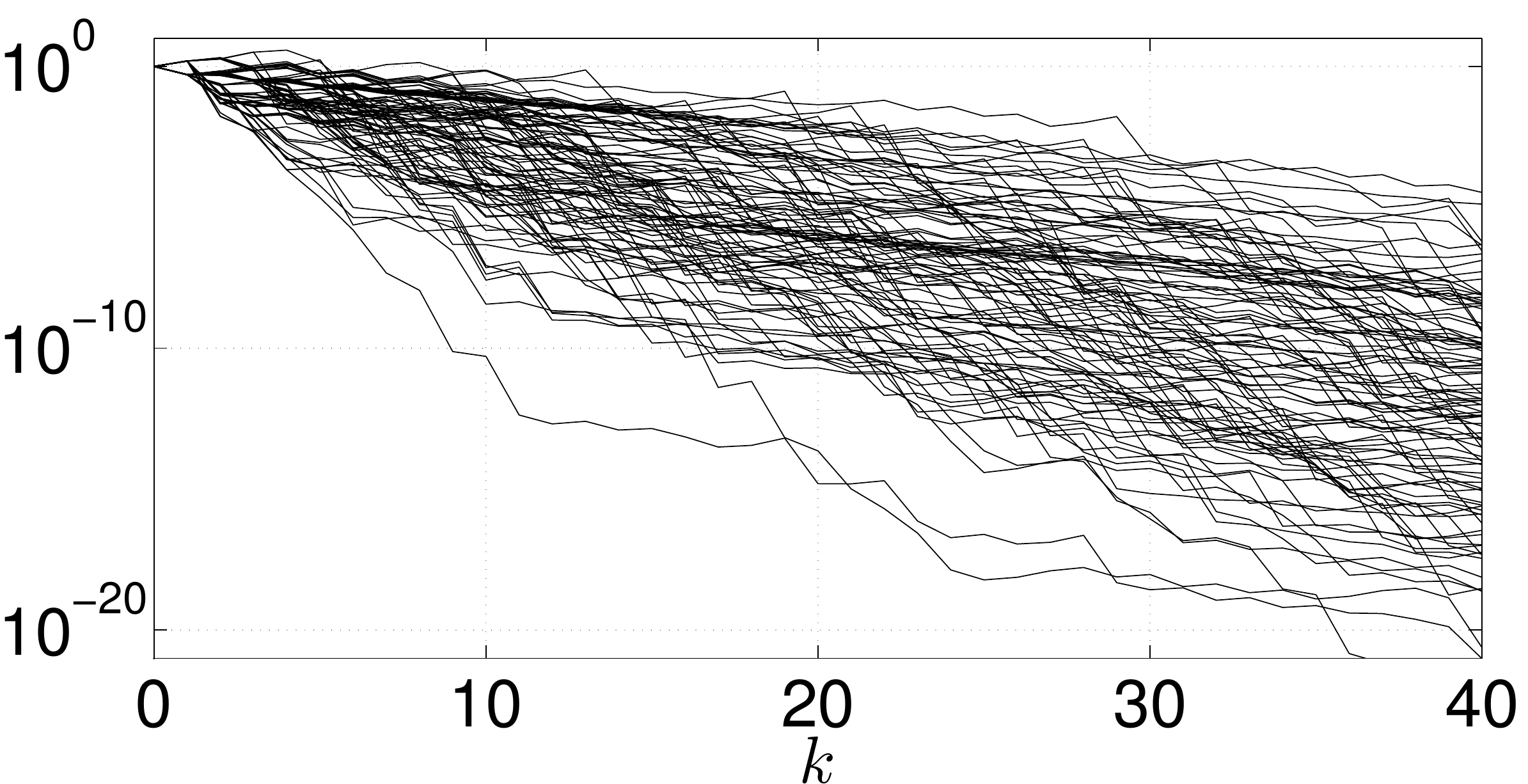}}
\caption{Stabilized system:  100 sample paths of ${\norm{x(k)}^2}$} \label{fig:graph3log}
\end{figure}
\vspace{-3mm}
\end{example}

\subsection[Observation at the initial time]{Observation at the Initial Time}

The controller designed by Theorem~\ref{thm:LMI} is stronger than the
ones in~\cite{Cetinkaya2013,Cetinkaya2013a,Cetinkaya2014} in the
following sense: though the controller by Theorem~\ref{thm:LMI} does
not necessarily need to know the Markov state at the initial time
$k=0$ for stabilization, the ones designed
in~\cite{Cetinkaya2013,Cetinkaya2013a,Cetinkaya2014} {are supposed to}
know the Markov state at $k=0$ {for stabilization}. The aim of this
subsection is to show {that the above two different types of
stabilization are in fact equivalent provided the first observation
time, $t_0$, is uniformly bounded}.

For a family of observation processes $\mathcal T$, we define another
family of observation processes $\mathcal T' \subset \mathcal T$ by
{\mbox{$\mathcal T'= \{t \in \mathcal T:  t_0 = 0\text{ with
probability one}\}$}.} $\mathcal T'$ expresses the set of all
observation processes in $\mathcal T$ that observes the Markov state
at $k=0$. In particular we can see that
\begin{equation}\label{eq:calT_s'=}
\mathcal T_{s, \Lambda}' 
=
\{t_{\Lambda}(s(\cdot;s_0))\}_{s_0\in \Lambda}.
\end{equation}
The difference from
\eqref{eq:def:T_s} is that the initial state~$s_0$ of the Markov
chain~$s$ is confined to be in $\Lambda$.

Then we can prove the following theorem. 

\begin{theorem}\label{thm:weak}
Assume that there exists $\tau > 0$ such that $t_0 \leq \tau$ with
probability one. Then $K$ stabilizes $(\Sigma, \mathcal T_{s,
\Lambda})$ if and only if $K$ stabilizes $(\Sigma, \mathcal T_{s,
\Lambda}')$.
\end{theorem}

\begin{sketchproof}{Sketch of the proof}
If $K$ stabilizes $(\Sigma, \mathcal T_{s, \Lambda})$ then $K$ clearly
stabilizes $(\Sigma, \mathcal T_{s, \Lambda}')$ because $\mathcal
T_{s, \Lambda}' \subset \mathcal T_{s, \Lambda}$. Next assume that $K$
stabilizes $(\Sigma, \mathcal T_{s, \Lambda}')$. Then. by the above
observation~\eqref{eq:calT_s'=}, there exist $C>0$ and $\epsilon\in
[0, 1)$ such that the solution $x$ of $\Sigma_K$ satisfies
\eqref{eq:MSS} for all \mbox{$x_0 \in \mathbb{R}^n$}, \mbox{$r_0 \in
\oneto N$}, and $s_0\in \Lambda$. Let us show that $K$ stabilizes
$(\Sigma, \mathcal T_{s, \Lambda})$. Let $x_0 \in \mathbb{R}^n$,
\mbox{$r_0 \in \oneto N$}, $s_0 \in \oneto M$, $\sigma_0 \in \oneto
N$, and $\tau_0 <0$ be arbitrary. By the assumption, there exist
\mbox{$p_0, \dotsc, p_\tau \geq 0$} such that {$\sum_{k=0}^\tau p_k =
1$} and $\Pr(t_0 = k_0) = p_{k_0}$. Fix a \mbox{$k_0 \in \{0, \dotsc,
\tau\}$} and consider the case $t_0 = k_0$. Let $C' = \max \{
\norm{A_i + B_i K_{\delta, \gamma}} \}_{i\in \oneto N, \delta\in
\oneto N, \gamma \in \oneto T}$, where $\norm{\cdot}$ denotes the
maximum singular value of a matrix. Then one can show
$E[\norm{x(k_0)}^2] \leq C'^{2k_0} \norm{x_0}^2$. {Since $x$ follows
the stabilized dynamics after $k=t_0$} we obtain $E[\norm{x(k)}^2]
\leq C \epsilon^{k-k_0} C'^{2k_0} \norm{x_0}^2$, which happens with
probability~$p_{k_0}$. Therefore, taking the summation of this
inequality  with respect to $k_0$ we can actually derive
\mbox{$E[\norm{x(k)}^2] \leq \sum_{k_0=0}^\tau p_{k_0} C
\epsilon^{k-k_0} C'^{2k_0} \norm{x_0}^2 \leq \left(\tau
CC'^{2\tau}\right) \epsilon^k \norm{x_0}^2$.} Hence $K$ stabilizes
$(\Sigma, \mathcal T_{s, \Lambda})$.
\end{sketchproof}

\section{Conclusion}

In this paper we studied the state-feedback stabilization of
discrete-time Markov jump linear systems when the observation of the
Markov-state by a controller is time-randomized by another Markov
chain. Using an embedding of the Markov state to another Markov chain,
we transformed the Markov jump linear system with time-randomized
observations to the one with clustered observations. Based on this
transformation we derived linear matrix inequalities for finding
state-feedback stabilizing gains. The proposed method can, in a
unified way, treat time-random observations including periodic and
renewal-type observations studied in the literature. A numerical
example is presented to show the effectiveness of the proposed method.



\begin{thebibliography}{10}
\providecommand{\url}[1]{#1}
\csname url@samestyle\endcsname
\providecommand{\newblock}{\relax}
\providecommand{\bibinfo}[2]{#2}
\providecommand{\BIBentrySTDinterwordspacing}{\spaceskip=0pt\relax}
\providecommand{\BIBentryALTinterwordstretchfactor}{4}
\providecommand{\BIBentryALTinterwordspacing}{\spaceskip=\fontdimen2\font plus
\BIBentryALTinterwordstretchfactor\fontdimen3\font minus
  \fontdimen4\font\relax}
\providecommand{\BIBforeignlanguage}[2]{{%
\expandafter\ifx\csname l@#1\endcsname\relax
\typeout{** WARNING: IEEEtran.bst: No hyphenation pattern has been}%
\typeout{** loaded for the language `#1'. Using the pattern for}%
\typeout{** the default language instead.}%
\else
\language=\csname l@#1\endcsname
\fi
#2}}
\providecommand{\BIBdecl}{\relax}
\BIBdecl

\bibitem{Siqueira2004}
A.~Siqueira and M.~H. Terra, ``{Nonlinear and Markovian $\mathcal H_\infty$
  controls of underactuated manipulators},'' \emph{IEEE Transactions on Control
  Systems Technology}, vol.~12, pp. 811--826, 2004.

\bibitem{Vargas2013a}
A.~N. Vargas, W.~Furloni, and J.~B. do~Val, ``{Second moment constraints and
  the control problem of Markov jump linear systems},'' \emph{Numerical Linear
  Algebra with Applications}, vol.~20, pp. 357--368, 2013.

\bibitem{Costa2007}
O.~L. Costa and W.~L. de~Paulo, ``{Indefinite quadratic with linear costs
  optimal control of Markov jump with multiplicative noise systems},''
  \emph{Automatica}, vol.~43, pp. 587--597, 2007.

\bibitem{Barthelemy2013a}
J.~Barth\'elemy and M.~Marx, ``{Monetary policy switching and indeterminacy},''
  \emph{Working paper}, 2013.

\bibitem{Hespanha2007a}
J.~P. Hespanha, P.~Naghshtabrizi, and Y.~Xu, ``{A survey of recent results in
  networked control systems},'' \emph{Proceedings of the IEEE}, vol.~95, pp.
  138--162, 2007.

\bibitem{Costa2005}
O.~Costa, M.~Fragoso, and R.~Marques, \emph{{Discrete-Time Markov Jump Linear
  Systems}}, ser. Probability and Its Applications.\hskip 1em plus 0.5em minus
  0.4em\relax London: Springer-Verlag, 2005.

\bibitem{DoVal2002}
%\BIBentryALTinterwordspacing
J.~B. do~Val, J.~C. Geromel, and A.~P. Gon\c{c}alves, ``{The $H_2$-control
  for jump linear systems: cluster observations of the Markov state},''
  \emph{Automatica}, vol.~38, pp. 343--349, 2002. 

\bibitem{Liu2006}
%\BIBentryALTinterwordspacing
H.~Liu, E.-K. Boukas, F.~Sun, and D.~W. Ho, ``{Controller design for Markov
  jumping systems subject to actuator saturation},'' \emph{Automatica},
  vol.~42, pp. 459--465, 2006. 

\bibitem{Li2006b}
D.~Li, D.~Zhang, and H.~Ji, ``{Stabilization of jump linear systems with
  partial observation of Markov mode},'' \emph{International Journal of Pure
  and Applied Mathematics}, vol.~27, pp. 31--38, 2006.

\bibitem{Vargas2013}
A.~N. Vargas, E.~F. Costa, and J.~B. do~Val, ``{On the control of Markov
  jump linear systems with no mode observation: application to a DC Motor
  device},'' \emph{International Journal of Robust and Nonlinear Control},
  vol.~23, pp. 1136--1150, 2013.

\bibitem{Cetinkaya2013}
A.~Cetinkaya and T.~Hayakawa, ``{Discrete-time switched stochastic control
  systems with randomly observed operation mode},'' in \emph{52nd IEEE
  Conference on Decision and Control}, 2013, pp. 85--90.

\bibitem{Cetinkaya2013a}
------, ``{Stabilizing discrete-time switched linear stochastic systems using
  periodically available imprecise mode information},'' in \emph{2013 American
  Control Conference}, 2013, pp. 3266--3271.

\bibitem{Cetinkaya2014}
%\BIBentryALTinterwordspacing
------, ``{Sampled-mode-dependent time-varying control strategy for stabilizing
  discrete-time switched stochastic systems},'' in \emph{2014 American Control
  Conference}, 2014, pp. 3966--3971. 

\bibitem{Bercu2009}
B.~Bercu, F.~Dufour, and G.~Yin, ``{Almost sure stabilization for feedback
  controls of regime-switching linear systems with a hidden Markov chain},''
  \emph{IEEE Transactions on Automatic Control}, vol.~54, pp. 2114--2125, 2009.

\end{thebibliography}
\end{document}